\documentclass[11pt,twoside]{article}
\usepackage[margin=1.3in]{geometry}

\usepackage{amsfonts,amsmath,amssymb,amsthm}
\usepackage{CJK,CJKnumb,CJKulem,cases,cite,color,curves}
\usepackage{bm,dsfont,enumitem,extarrows,fontenc,graphicx,ifthen,latexsym,listings,makeidx,mathrsfs,psfrag,times,xcolor}
\usepackage[colorlinks,citecolor=green,linkcolor=red]{hyperref}
\usepackage[title]{appendix}

\let\oldbibliography\thebibliography
\renewcommand{\thebibliography}[1]{\oldbibliography{#1}\setlength{\itemsep}{0pt}}

\makeindex
\newtheorem{theorem}{Theorem}[section]

\newcommand{\R}{\mathbb R}

\newcommand{\Sp}{\mathbb S}

\begin{document}

\title{\textbf{Blow up limits of the fractional Laplacian and their applications to the fractional Nirenberg problem}\bigskip}

\author{\medskip Xusheng Du, \quad Tianling Jin\footnote{T. Jin was partially supported by Hong Kong RGC grant GRF 16302217.}, \quad Jingang Xiong\footnote{J. Xiong was is partially supported by NSFC grants 11922104 and 11631002.}, \quad Hui Yang}

\date{\today}

\maketitle

\begin{abstract}
We show a convergence result of the fractional Laplacian for sequences of nonnegative functions without uniform boundedness near infinity. As an application, we construct a sequence of solutions to the fractional Nirenberg problem that blows up in the region where the prescribed functions are negative. This is a different phenomenon from the classical Nirenberg problem. 
\medskip


\medskip


\end{abstract}

\section{Introduction}
Let $n\ge 1$ be an integer. In the blow up analysis for equations with the Laplacian operator $\Delta$,  the following trivial fact plays an essential role: 
\begin{center}
``\it{If $u_i\to u$  in $C_{loc}^2 (\R^n)$  as $i\to\infty$, then $\Delta u_i\to\Delta u$   in $C_{loc} (\R^n)$.}"
\end{center}
In this paper, we first show a similar but different property for the fractional Laplacian operator. 

Let $\sigma \in (0, 1)$. The fractional Laplacian $(-\Delta)^\sigma$ is defined by
\[
(-\Delta)^\sigma u(x)=c_{n,\sigma} {\rm P.V.} \int_{\R^n}\frac{u(x)-u(y)}{|x-y|^{n+2\sigma}}\,dy 
\]
with $c_{n, \sigma} = \frac{ 2^{2 \sigma} \sigma \Gamma \left( \frac{n + 2 \sigma}{2} \right) }{ \pi^{n/2} \Gamma (1 - \sigma) }$ and the gamma function $\Gamma$.  Define
\[
L_\sigma (\R^n) = \left\{ u \in L_{loc}^1 (\R^n) : \int_{\R^n} \frac{|u(x)|}{ 1 + |x|^{n + 2 \sigma} } dx < \infty \right\}.
\]
Then $(- \Delta)^\sigma u (x)$ is well-defined if $u \in L_\sigma (\R^n)$ and $u$ is $C^{2\sigma+\alpha}$ near $x$ for some $\alpha>0$.  

\begin{theorem}\label{thm:convergence}
Let $n \ge 1$, $\sigma \in (0, 1)$ and $\alpha > 0$. Suppose that $\{ R_i \}$ is a sequence of positive numbers converging to $+ \infty$, and $\{ u_i \} \subset L_\sigma (\R^n) \cap C^{2 \sigma + \alpha} (B_{R_i})$ is a sequence of nonnegative functions. If $\{ u_i \}$ converges in $C_{loc}^{2 \sigma + \alpha} (\R^n)$ to a function $u\in L_\sigma(\R^n)$, and $\{ (- \Delta)^\sigma u_i \}$ converges pointwisely in $\R^n$, then there exists a constant $b \geq 0$ such that
\begin{equation}\label{eq:witha}
\lim_{i \to \infty} (- \Delta)^\sigma u_i (x) = (- \Delta)^\sigma u(x) - b ~~~~~~ \forall \, x \in \R^n.
\end{equation}
Moreover, 
\[
b=c_{n,\sigma}\lim_{R\to\infty}\lim_{i\to\infty} \int_{B_R^c} \frac{u_i (x)}{ |x|^{n + 2 \sigma } } \,dx\quad\mbox{(the limit exists and is finite)}.
\]
\end{theorem}

Note that Theorem \ref{thm:convergence} does not assume uniform boundedness of $\{ u_i \} $ near infinity in any topology, but rather their non-negativity instead. The following example shows that the constant $b$ in \eqref{eq:witha} can be positive. 

\begin{theorem}\label{thm:convergencenonzero}
Let $n\ge 1$ and $\sigma \in (0, 1)$. There exists a sequence of nonnegative smooth functions $\{ v_j \}$ with compact support in $\R^n$ such that $v_j $ converges to $1$ in $C_{loc}^2 (\R^n)$, and 
$$
\lim_{j \to \infty} (- \Delta)^\sigma v_j (x) = -1 ~~~~~~ \forall \, x \in \R^n.
$$
\end{theorem}

The nonzero constant $b$ in Theorem \ref{thm:convergence} will lead to some different blow up phenomena for equations with the fractional Laplacian, compared with those involving the Laplacian. 

In \cite{GS2}, Gidas-Spruck used blow up analysis, together with Liouville theorems that were proved by them earlier in \cite{GS} , to derive a priori estimate of positive solutions to general elliptic equations including
\begin{equation}\label{SecOr01}
- \Delta  u = K(x) u^p, 
\end{equation}
where $p$ is a subcritical Sobolev exponent, i.e., $1<p<\infty$ if $n=1,2$, or $1<p<\frac{n+2}{n-2}$ if $n\ge 3$.

When $n\ge 3$ and $p=\frac{n+2}{n-2}$, which is the critical Sobolev case, the equation \eqref{SecOr01} is the one for the Nirenberg problem, that is the problem of finding a metric conformal to the flat metric on $\mathbb{R}^n$ such that $K(x)$ is the scalar curvature of the new metric.  One main ingredient in establishing the existence of solutions to the Nirenberg problem is to obtain a priori estimates for the solutions of \eqref{SecOr01}, which have been achieved in Chang-Gursky-Yang \cite{CGY93}, Li \cite{Li95} and Schoen-Zhang \cite{SZ96} for positive functions $K$, and in Chen-Li \cite{CL97} and Lin \cite{Lin98} for $K$ changing signs. In particular, when $K$ is allowed to change signs, one important step is to estimate the solutions in the region where $K$ is negative. It has been shown in Chen-Li \cite{CL97} and Lin \cite{Lin98} that the solutions of \eqref{SecOr01} in $B_2$ are uniformly bounded in $B_{1}$ provided that $K \in C^\alpha(B_2)$ for some $\alpha \in(0,1)$  and $K(x) \leq -c_0$ in $B_2$ for some $c_0>0$. The uniform bound depends only on $n$, $c_0$ and $\|K\|_{C^\alpha(B_2)}$.

On the contrary, we will show that such a priori estimates do not hold for the fractional Nirenberg problem. The fractional Nirenberg problem is equivalent to studying the equations  
\begin{equation}\label{eq:sign}
(- \Delta)^\sigma u = K(x) u^p ~~~~~~ \textmd{in} ~ \R^n.
\end{equation} 
We will construct a sequence of positive smooth solutions to \eqref{eq:sign} that blows up in the region where $K$ is negative. Note that a priori estimates of the fractional equation \eqref{eq:sign} for positive functions $K$ have been derived in Jin-Li-Xiong \cite{JLX14,JLX15,JLX17}.  More generally, blow up can occur for the equation \eqref{eq:sign} with any $p \in \R$.

\begin{theorem}\label{thm:blowup}
Let $n\ge 1$, $\sigma \in (0, 1)$, $p \in \R$ and $q > - 2 \sigma$. There exist two positive constants $c$ and $C$ depending only on $n$, $\sigma$, $p$ and $q$, a family of functions $\{ K_\lambda \}_{\lambda \geq 1} \subset C^\infty (\R^n)$ satisfying
$$
- C \leq K_\lambda (x) \leq - c, \ \ c \leq | \nabla K_\lambda (x) | \leq C \ \ \mbox{and} \ \ | \nabla^2 K_\lambda (x) | \leq C \quad \forall \, x \in B_2 \ \mbox{and} \ \forall \, \lambda \geq 1,
$$
and a family of positive functions $\{ u_\lambda \}_{\lambda \geq 1} \subset C^\infty (\R^n)$ satisfying
$$
(- \Delta)^\sigma u_\lambda = K_\lambda (x) u_\lambda^p \quad \mbox{in} ~ \R^n, \qquad \lim_{|x| \to \infty} |x|^q u_\lambda (x) = 1,
$$
and
$$
\min_{ \overline{B}_1 } u_\lambda \to + \infty \quad \mbox{ as } \lambda \to + \infty.
$$
\end{theorem}

The failure of the compactness of solutions to \eqref{eq:sign} in the region where $K$ is negative is caused by Theorem \ref{thm:convergence} and Theorem \ref{thm:convergencenonzero} that the blow up limit equation for the fractional equation \eqref{eq:sign} in the region $\{K(x) \leq -c_0 \}$ is
\begin{equation}\label{eq:globalnegative}
(-\Delta)^\sigma u -b=-u^{p}\quad\mbox{in }\R^n
\end{equation}
with a constant $b\ge 0$, and $b$ can be positive. If $b>0$, then the equation \eqref{eq:globalnegative} clearly has the positive constant $b^{\frac1p}$ as a solution. Such nonzero constant solutions of \eqref{eq:globalnegative} cannot be ruled out if one does not know a priori uniform boundedness in any topology.

This paper is organized as follows.  In Section \ref{Sect3}, we show the convergence for the fractional Laplacian in Theorem \ref{thm:convergence} and give the example in 
Theorem \ref{thm:convergencenonzero}. In Section \ref{Sect02}, we construct blow up solutions for the fractional Nirenberg problem stated in Theorem \ref{thm:blowup}.

\section{Convergence for the fractional Laplacian}\label{Sect3}

\begin{proof}[Proof of Theorem \ref{thm:convergence}]
Our proof is inspired by that of \cite[Proposition 2.9]{JLX17} on an integral equation.  Fix $x \in \R^n$. Let $R \gg |x| + 1$. Then for all large $i$, we have
\begin{equation}\label{AEF} 
\aligned
(- \Delta)^\sigma u(x) - (- \Delta)^\sigma u_i (x) = & ~ c_{n, \sigma} \int_{B_R} \frac{ (u - u_i) (x) - (u - u_i) (y) }{ |x - y|^{n + 2 \sigma} } d y \\
& + c_{n, \sigma} \int_{B_R^c} \frac{ (u - u_i) (x) - u(y) }{ |x - y|^{n + 2 \sigma} } d y \\
& + c_{n, \sigma} \int_{B_R^c} \frac{u_i (y)}{ |x - y|^{n + 2 \sigma} } d y \\
= : & ~ A_i (x, R) + E_i (x, R) + F_i (x, R).
\endaligned
\end{equation}
By the assumptions, $\{ u_i \}$ converges to $u$ in $C^{2 \sigma + \alpha} (B_{2R})$. Hence,
\begin{equation}\label{eq:limitA}
\lim_{i \to \infty} A_i (x, R) = 0 ~~~~~~ \textmd{and} ~~~~~~ \lim_{i \to \infty} E_i (x, R) = - c_{n, \sigma} \int_{B_R^c} \frac{u(y)}{ |x - y|^{n + 2 \sigma} } d y.
\end{equation}
Since $u\in L_\sigma(\R^n)$, we have
\begin{equation}\label{eq:limitE}
\lim_{R\to+\infty}\left| \lim_{i \to \infty} E_i (x, R) \right|=0.
\end{equation}
Since the sequence $\{ (- \Delta)^\sigma u_i  (x) \}$ converges, it follows from \eqref{AEF} that
$$
F(x, R) : = \lim_{i \to \infty} F_i (x, R) \quad \mbox{exists and is finite}.
$$
Sending $i \to \infty$ in \eqref{AEF}, and using \eqref{eq:limitA}-\eqref{eq:limitE}, we obtain
\begin{equation}\label{CyR}
\lim_{R\to+\infty} \left| (- \Delta)^\sigma u(x) - \lim_{i \to \infty} (- \Delta)^\sigma u_i (x) - F(x, R) \right| =0.
\end{equation}

Since $u_i$ is nonnegative on $\R^n$, we have
$$
\left( \frac{R}{R + |x|} \right)^{n + 2 \sigma} \int_{B_R^c} \frac{u_i (y)}{ |y|^{n + 2 \sigma} } d y \leq \int_{B_R^c} \frac{u_i (y)}{ |x - y|^{n + 2 \sigma} } d y \leq \left( \frac{R}{R - |x|} \right)^{n + 2 \sigma} \int_{B_R^c} \frac{u_i (y)}{ |y|^{n + 2 \sigma} } d y.
$$
Sending $i\to\infty$, it follows that
\begin{equation}\label{eq:limitF}
\left( \frac{R}{R + |x|} \right)^{n + 2 \sigma} F(0, R) \leq F(x, R) \leq \left( \frac{R}{R - |x|} \right)^{n + 2 \sigma} F(0, R).
\end{equation}
Notice that $F_i (x, R)$ is nonnegative and non-increasing in $R$, so is $F(x, R)$. Hence, $\displaystyle\lim_{R \to \infty} F(x, R)$ exists, and is nonnegative and finite. By sending $R$ to $\infty$ in \eqref{eq:limitF}, we obtain
$$
\lim_{R \to \infty} F(x, R) = \lim_{R \to \infty} F(0, R) = : b \geq 0.
$$
Then the conclusion follows from \eqref{CyR}.
\end{proof}

The proof of Theorem \ref{thm:convergencenonzero} can be illustrated by the following simple functions. Let $\lambda > 0$. Define
\[
W_\lambda (x) = \left\{
\aligned
& \lambda ~~~~~~ & \textmd{if} & ~ x \in B_3, \\
& \lambda + \lambda^2 ~~~~~~ & \textmd{if} & ~ x \in B_3^c.
\endaligned
\right.
\]
For $j \ge 1$, let
$$
R_j = j^\frac{1}{2 \sigma},
$$
and
$$
V_j (x) = j^{- 1} W_j \left( j^{ - \frac{1}{2 \sigma} } x \right) = \left\{
\aligned
& 1 ~~~~~~ & \textmd{if} & ~ y \in B_{3 R_j}, \\
& 1 + j ~~~~~~ & \textmd{if} & ~ y \in B_{3 R_j}^c.
\endaligned
\right.
$$
Then $V_j$ converges to the constant function $1$ in $C_{loc}^{2} (\R^n)$. However, for any $x \in \R^n$, 
$$
\lim_{j\to\infty}(- \Delta)^\sigma V_j (x) = - c_{n, \sigma} \lim_{j\to\infty}\int_{B_3^c} \frac{d y}{ | j^{ - \frac{1}{2 \sigma} } x - y |^{n + 2 \sigma} } = - \frac{ c_{n, \sigma} |\Sp^{n - 1}| }{ 2 \sigma 3^{2 \sigma} }.
$$

A  mollification of $V_j$ by smoothing it out and cutting it off at infinity will give the example for Theorem \ref{thm:convergencenonzero}. The details are as follows.

\begin{proof}[Proof of Theorem \ref{thm:convergencenonzero}]
Let $\eta : \R \to [0, 1]$ be a $C^\infty$ cut-off function such that $\eta (t) \equiv 0$ if $t \leq 0$, $\eta (t) \equiv 1$ if $t \geq 1$, and $0 \leq \eta (t) \leq 1$ if $0 \leq t \leq 1$. Let $\psi (x): = \eta (|x| - 3)$ and $\varphi (x): = \eta (|x| - 6)$.

For $\lambda \geq 1$, define
$$
w_\lambda (x) = \left\{
\aligned
& \lambda ~~~~~~ & \textmd{if} & ~ x \in B_3, \\
& \lambda + \lambda^2 \psi (x) ~~~~~~ & \textmd{if} & ~ x \in B_4 \setminus B_3, \\
& \lambda + \lambda^2 ~~~~~~ & \textmd{if} & ~ x \in B_6 \setminus B_4, \\
& ( 1 - \varphi (x) ) ( \lambda + \lambda^2 ) ~~~~~~ & \textmd{if} & ~ x \in B_6^c.
\endaligned
\right.
$$
Then $w_\lambda \in C_c^\infty (\R^n)$, $w_\lambda\ge 0$ in $\R^n$, and ${\rm supp} \, (w_\lambda) \subset \overline{B}_{7}$ for any $\lambda \geq 1$.

Define
$$
f_\lambda (x) = \lambda^{- 2} (- \Delta)^\sigma w_\lambda (x), ~~~~~~ x \in \R^n.
$$
Then for $x \in B_2$,
$$
\aligned
f_\lambda (x) = & ~ - c_{n, \sigma} \int_{B_4 \setminus B_3} \frac{\psi (y)}{ |x - y|^{n + 2 \sigma} } d y - c_{n, \sigma} \int_{B_6 \setminus B_4} \frac{d y}{ |x - y|^{n + 2 \sigma} } \\
& ~ + \lambda^{- 1} c_{n, \sigma} \int_{ B_6^c } \frac{\varphi (y)}{ |x - y|^{n + 2 \sigma} } d y - c_{n, \sigma} \int_{ B_6^c } \frac{1 - \varphi (y)}{ |x - y|^{n + 2 \sigma} } d y\\
& ~ \to - c_{n, \sigma} \left( \int_{B_4 \setminus B_3} \frac{\psi (y)}{ |x - y|^{n + 2 \sigma} } d y + \int_{B_6 \setminus B_4} \frac{d y}{ |x - y|^{n + 2 \sigma} } + \int_{ B_6^c } \frac{1 - \varphi (y)}{ |x - y|^{n + 2 \sigma} } d y \right)
\endaligned
$$
uniformly as $\lambda$ tends to $\infty$.

Let 
\begin{equation*}
\beta= \left(- \lim_{\lambda \to \infty} f_\lambda (0)\right)^{- \frac{1}{2 \sigma}} .
\end{equation*}
For $ j \ge 1$, let
$$
R_j = \beta^{- 1} j^\frac{1}{2 \sigma} \quad\mbox{and}\quad v_j (x) = j^{- 1} w_j \left( \beta j^{ - \frac{1}{2 \sigma} } x \right).
$$
Then $v_j\in C_c^\infty(\R^n)$, $ v_j = 1$ in $B_{R_j}$, $ v_j \ge 0$ in $\R^n$, and
$$
(- \Delta)^\sigma v_j (x) = \beta^{2 \sigma } f_j \left( \beta j^{ - \frac{1}{2 \sigma} } x \right) ~~~~~~ \textmd{for} ~ x \in B_{R_j}.
$$
It follows that $v_j$ converges to the constant function $1$ in $C^2_{loc} (\R^n)$ as $j\to \infty$, and 
$$
\lim_{j \to \infty} (- \Delta)^\sigma v_j (x) = \beta^{2 \sigma} \lim_{j \to \infty} f_j (\beta j^{ - \frac{1}{2 \sigma} } x ) = - 1\quad \forall\, x \in \R^n.
$$
The proof is completed.
\end{proof}

\section{Blow up phenomena for a fractional Nirenberg problem when the prescribed functions are negative}\label{Sect02}

Our example to Theorem \ref{thm:blowup} is inspired by the following simple functions. Let $\lambda > 0$. Define
$$
U_\lambda (x) = \left\{
\aligned
& \lambda ~~~~~~ & \textmd{if} & ~ x \in B_3, \\
& \lambda + \lambda^p ~~~~~~ & \textmd{if} & ~ x \in B_3^c.
\endaligned
\right.
$$
Then for $x \in B_2$,
$$
\frac{ (- \Delta)^\sigma U_\lambda (x) }{U_\lambda (x)^p} = - c_{n, \sigma} \int_{B_3^c} \frac{d y}{ |x - y|^{n + 2 \sigma} } = : K(x).
$$
It is clear that $K$ is smooth in $B_2$, and bounded from above and below by two negative constants in $B_2$. Nevertheless,
$$
\min_{ x \in \overline{B}_1 } U_\lambda (x) \to + \infty ~~~~~~ \textmd{as} ~ \lambda \to + \infty.
$$

A similar mollification of $U_\lambda$ by smoothing it out and cutting it off at infinity will give the example for Theorem \ref{thm:blowup}. The details are as follows.

\begin{proof}[Proof of Theorem \ref{thm:blowup}] Let $\eta : \R \to [0, 1]$ be a $C^\infty$ cut-off function such that $\eta (t) \equiv 0$ if $t \leq 0$, $\eta (t) \equiv 1$ if $t \geq 1$, and $0 \leq \eta (t) \leq 1$ if $0 \leq t \leq 1$. For $\lambda \geq 1$, we define
$$
u_\lambda (x) = \left\{
\aligned
& \lambda ~~~~~~ & \textmd{if} & ~ x \in B_3, \\
& \lambda + \lambda^p \psi (x)  ~~~~~~ & \textmd{if} & ~ x \in B_4 \setminus B_3, \\
& \lambda + \lambda^p ~~~~~~ & \textmd{if} & ~ x \in B_R \setminus B_4, \\
& ( 1 - \varphi (x)  ) ( \lambda + \lambda^p ) + \varphi (x)  |x|^{- q} ~~~~~~ & \textmd{if} & ~ B_R^c,
\endaligned
\right.
$$
where $R = R (n,\sigma, p, q, \lambda) > 9$ is a large constant to be specified later, $\psi (x) : = \eta (|x| - 3)$ and $\varphi (x) : = \eta (|x| - R)$. Since $q > - 2 \sigma$, $u_\lambda \in C^\infty (\R^n) \cap L_\sigma (\R^n)$ and $u_\lambda>0$ in $\R^n$. Define
$$
K_\lambda (x) = \frac{ (- \Delta)^\sigma u_\lambda (x) }{u_\lambda (x)^p}, ~~~~~~ x \in \R^n.
$$
Then $K_\lambda \in C^\infty (\R^n)$.

\vskip0.1in

{\it Step 1. There exist two positive constants $c_1$ and $c_2$ depending only on $n$ and $\sigma$ such that}
\begin{equation}\label{eq:c1Kc2}
- c_1 \leq K_\lambda (x) \leq - c_2 ~~~~~~ \forall \, x \in B_2, ~ \lambda \geq 1.
\end{equation}

Indeed, for $x \in B_2$, we have
\begin{equation}\label{eq:B2K}
\aligned
c_{n, \sigma}^{- 1} K_\lambda (x) = & ~ - \int_{B_4 \setminus B_3} \frac{\psi (y)}{ |x - y|^{n + 2 \sigma} } d y - \int_{B_R \setminus B_4} \frac{d y}{ |x - y|^{n + 2 \sigma} } \\
& ~ +  \int_{B_R^c} \frac{\lambda^{1 - p}\varphi (y)+\varphi (y)-1 - \lambda^{- p} \varphi (y) |y|^{- q}}{ |x - y|^{n + 2 \sigma} } d y \\
= : & ~ \sum_{i = 1}^3 I_i.
\endaligned
\end{equation}
For convenience, denote
$$
\gamma_{n, \sigma} : = \int_{B_1^c} \frac{d y}{ |y|^{n + 2 \sigma} }=\frac{ |\Sp^{n - 1}| }{2 \sigma}, \quad \widetilde{\gamma}_{n, \sigma, q} : = \int_{B_1^c} \frac{d y}{ |y|^{n + 2 \sigma+q} }= \frac{|\Sp^{n - 1}|}{2 \sigma + q}.
$$
It is clear that 
$$
0\ge I_1 \geq - \int_{B_1 (x)^c} \frac{d y}{ |x - y|^{n + 2 \sigma} } = - \int_{B_1^c} \frac{d y}{ |y|^{n + 2 \sigma} } = - \gamma_{n, \sigma}.
$$
Similarly, we obtain
\[
 - \gamma_{n, \sigma} \le I_2 \leq - \int_{B_{R - 2} \setminus B_6} \frac{d y}{ |y|^{n + 2 \sigma} } = - \gamma_{n, \sigma} \left( 6^{- 2 \sigma} - (R - 2)^{- 2 \sigma} \right).
\]
Since $R>9$, we have $|x-y|\ge |y|/2$ for all $x\in B_2$ and $y\in B_R^c$. Thus,
\begin{align*}
|I_3|&\le  2^{n+2\sigma}\int_{B_{R}^c} \frac{(\lambda^{1 - p}+1 +\lambda^{-p} |y|^{-q})d y}{ |y|^{n + 2 \sigma} } \\
&= 2^{n+2\sigma}  \gamma_{n, \sigma}  (\lambda^{1 - p}+1) R^{- 2 \sigma}+ 2^{n+2\sigma}  \widetilde\gamma_{n, \sigma,q}  \lambda^{-p} R^{- 2 \sigma-q}.
\end{align*}
By choosing $R = R (n, \sigma, p, q, \lambda) > 9$ sufficiently large such that
$$
\gamma_{n, \sigma} (R - 2)^{- 2 \sigma} +2^{n+2\sigma}  \gamma_{n, \sigma}  (\lambda^{1 - p}+1) R^{- 2 \sigma}+ 2^{n+2\sigma}  \widetilde\gamma_{n, \sigma,q}  \lambda^{-p} R^{- 2 \sigma-q}\le \frac{\gamma_{n, \sigma} 6^{-2\sigma}}{2},
$$
we have for all $x \in B_2$ and $\lambda \geq 1$ that
\[
- 3c_{n, \sigma} \gamma_{n, \sigma} \leq K_\lambda (x) \leq - \frac{ c_{n, \sigma} \gamma_{n, \sigma} 6^{- 2 \sigma} }{2}.
\]
Step 1 is proved.

\vskip0.1in

{\it Step 2. Moreover, by differentiating \eqref{eq:B2K}, we can similarly obtain
\begin{equation}\label{eq:D2Kc3}
| \nabla K_\lambda (x) | + | \nabla^2 K_\lambda (x) | + | \nabla^3 K_\lambda (x) | \leq c_3 ~~~~~~ \forall \, x \in B_2, ~ \lambda \geq 1,
\end{equation}
for some constant $c_3 > 0$ depending only on $n$ and $\sigma$.}

\vskip0.1in

{\it Step 3. There exists a constant $c_4 > 0$ depending only on $n$ and $\sigma$ such that
\begin{equation}\label{eq:HessK}
\nabla^2 K_\lambda (0) \leq - c_4 {\bf I}_n ~~~~~~ \forall \, \lambda \geq 1,
\end{equation}
where ${\bf I}_n$ is the $n \times n$ identity matrix.}

Indeed, for $y \in B_3^c$, we have
$$
\frac{\partial^2}{\partial x_i \partial x_j} \left( |x - y|^{- n - 2 \sigma} \right) (0) = (n + 2 \sigma) |y|^{- n - 2 \sigma - 4} \left[ (n + 2 \sigma + 2) y_i y_j - \delta_{ij} |y|^2 \right].
$$
It follows from \eqref{eq:B2K} that
\begingroup
\allowdisplaybreaks
\begin{align*}
& [ (n + 2 \sigma) c_{n, \sigma} ]^{- 1} \frac{ \partial^2 K_\lambda }{\partial x_i \partial x_j} (0) \\
&=   - \int_{B_4 \setminus B_3} \frac{\psi (y)}{ |y|^{n + 2 \sigma + 4} } \left[ (n + 2 \sigma + 2) y_i y_j - \delta_{ij} |y|^2 \right] d y \\
& \quad - \int_{B_R \setminus B_4} \frac{1}{ |y|^{n + 2 \sigma + 4} } \left[ (n + 2 \sigma + 2) y_i y_j - \delta_{ij} |y|^2 \right] d y \\
& \quad + \lambda^{1 - p} \int_{B_R^c} \frac{\varphi (y)}{ |y|^{n + 2 \sigma + 4} } \left[ (n + 2 \sigma + 2) y_i y_j - \delta_{ij} |y|^2 \right] d y \\
& \quad - \int_{B_R^c} \frac{1- \varphi (y)+ \lambda^{- p} \varphi (y) |y|^{- q} }{ |y|^{n + 2 \sigma + 4} } \left[ (n + 2 \sigma + 2) y_i y_j - \delta_{ij} |y|^2 \right] d y \\
= : & ~ \sum_{l = 1}^4 E_{l, ij}.
\end{align*}
\endgroup
Using the polar coordinate, and the radial symmetry of $\psi$ and $\varphi$, we have
$$
\aligned
E_{1, ii} & = - \int_{B_4 \setminus B_3} \frac{\psi (y)}{ |y|^{n + 2 \sigma + 4} } \left[ \left( \frac{n + 2 \sigma + 2}{n} \right) |y|^2 - |y|^2 \right] d y  \leq 0.
\endaligned
$$
Similarly, $E_{4, ii} \le 0$,
$$
E_{2, ii} = - \frac{2 (\sigma + 1)}{n} \int_{B_R \setminus B_4} \frac{d y}{ |y|^{n + 2 \sigma + 2} } = - |B_1| \left( 4^{- 2 \sigma - 2} - R^{- 2 \sigma - 2} \right),
$$
and
$$
E_{3, ii} \leq \lambda^{1 - p} \frac{2 (\sigma + 1)}{n} \int_{B_R^c} \frac{d y}{ |y|^{n + 2 \sigma + 2} } = \lambda^{1 - p} |B_1| R^{- 2 \sigma - 2}.
$$
Consequently, we obtain
$$
\frac{ \partial^2 K_\lambda }{\partial x_i^2} (0) \leq - (n + 2 \sigma) c_{n, \sigma} |B_1| \left( 4^{- 2 \sigma - 2} - (1 + \lambda^{1 - p}) R^{- 2 \sigma - 2} \right).
$$
We take $R = R (n, \sigma, p, q, \lambda) > 9$ sufficiently large such that
\begin{equation*}
4^{- 2 \sigma - 2} - (1 + \lambda^{1 - p}) R^{- 2 \sigma - 2} \geq 4^{- 2 \sigma - 3}.
\end{equation*}
Then for all $\lambda \geq 1$ and $1 \leq i \leq n$,
$$
\frac{ \partial^2 K_\lambda }{\partial x_i^2} (0) \leq - (n + 2 \sigma) c_{n, \sigma} |B_1| 4^{- 2 \sigma - 3}.
$$
By the symmetry of the regions in the above integrals, we know that
$$
\frac{ \partial^2 K_\lambda }{\partial x_i \partial x_j} (0) = 0 ~~~~~~ \textmd{if} ~ i \neq j.
$$
This proves Step 3.

\vskip0.1in

{\it Step 4. Construct the desired functions.} Notice that $u_\lambda$ is radially symmetric about the origin, so is $K_\lambda$. Hence,
\begin{equation}\label{eq:DK0=0}
\nabla K_\lambda (0) = 0.
\end{equation}
Then it follows from \eqref{eq:D2Kc3}, \eqref{eq:HessK} and \eqref{eq:DK0=0} that there exist two constants $\delta_0 \in (0, 1/4)$ and $c_5 > 0$ depending only on $n$ and $\sigma$ such that
\begin{equation}\label{eq:DKc5}
| \nabla K_\lambda (x) | \geq c_5 ~~~~~~ \forall \, x \in B_{2 \delta_0} (4 \delta_0 e_1), ~ \lambda \geq 1,
\end{equation}
where $e_1 = (1, 0, \dots, 0) \in \R^n$.

Let
$$
\widetilde{u}_\lambda (x) = \delta_0^q u_\lambda \left( \delta_0 (x + 4 e_1) \right) \quad \mbox{and} \quad \widetilde{K}_\lambda (x) = \delta_0^{q + 2 \sigma - p q} K_\lambda \left( \delta_0 (x + 4 e_1) \right).
$$
Then,
$$
(- \Delta)^\sigma \widetilde{u}_\lambda = \widetilde{K}_\lambda (x) \widetilde{u}_\lambda^p ~~~~~ \textmd{in} ~ \R^n.
$$
It follows from \eqref{eq:c1Kc2}, \eqref{eq:D2Kc3} and \eqref{eq:DKc5} that
$$
- C \le \widetilde{K}_\lambda (x) \leq - c, \ \ c \le | \nabla \widetilde{K}_\lambda(x) | \leq C \ \ \textmd{and} \ \ | \nabla^2 \widetilde{K}_\lambda(x) | \leq C \quad \forall\, x\in B_2 \mbox{ and } \forall \, \lambda \geq 1
$$
for some positive constants $c$ and $C$ depending only on $n$, $\sigma$, $p$ and $q$. Meanwhile, it follows from the definition of $u_\lambda$ that
$$
\lim_{|x| \to \infty} |x|^q \widetilde{u}_\lambda (x) = 1, 
$$
and
$$
\min_{\overline{B}_1} \widetilde{u}_\lambda = \delta_0^q \lambda \to + \infty \quad \mbox{ as }\lambda \to + \infty.
$$
The proof of Theorem \ref{thm:blowup} is completed.
\end{proof}

\bigskip

\noindent X. Du, \ \  T. Jin, \ \  H. Yang

\noindent Department of Mathematics, The Hong Kong University of Science and Technology\\
Clear Water Bay, Kowloon, Hong Kong\\[1mm]
Emails: \textsf{xduah@connect.ust.hk},\ \  \textsf{tianlingjin@ust.hk},\ \  \textsf{mahuiyang@ust.hk}

\medskip

\noindent J. Xiong

\noindent School of Mathematical Sciences, Laboratory of Mathematics and Complex Systems, MOE\\ Beijing Normal University, 
Beijing 100875, China\\[1mm]
Email: \textsf{jx@bnu.edu.cn}

\end{document}